\documentclass[ 12pt,reqno]{amsart}

\usepackage{graphicx}
\usepackage{amsmath, amsfonts, amssymb, amsthm, xspace}
\usepackage{xcolor}
\usepackage[margin=0.75in]{geometry}

\usepackage{booktabs} 
\usepackage{enumitem} 
\makeatletter
\def\namedlabel#1#2{\begingroup
    #2%
    \def\@currentlabel{#2}%
    \phantomsection\label{#1}\endgroup
}
\makeatother

\numberwithin{equation}{section}

\usepackage[utf8]{inputenc}

\renewcommand{\pmod}[1]{\left( \mathrm{ mod\;}#1\right)}

\usepackage{mathtools}

\usepackage{mathrsfs}

\newcommand{\scrD}{\mathscr D}

\usepackage{oldgerm}

\newcommand{\SF}{\mathfrak{SF}}

\renewcommand{\SF}{\mathfrak{S}}
\renewcommand{\SF}{\mathscr S}

\newcommand{\cD}{\mathcal{D}}
\newcommand{\cF}{\mathcal{F}}

\newcommand{\cR}{\mathcal R}
\newcommand{\cT}{\mathcal T}

\newcommand{\cV}{\mathcal V}
\newcommand{\cW}{\mathcal W}

\newcommand{\cE}{\mathcal E}

\newcommand{\ZZ}{\mathbb{Z}}

\newcommand{\NN}{\mathbb{N}}
\newcommand{\QQ}{\mathbb{Q}}

\newcommand{\SL}{\operatorname{SL}(2,\NN)}

\newcommand{\va}{\bar{a}}
\newcommand{\vq}{\bar{q}}
\newcommand{\vva}{\bar{\bar{a}}}
\newcommand{\vvq}{\bar{\bar{q}}}

\newcommand{\vvvq}{\bar{\bar{\bar{q}}}}

\newcommand{\mystar}{\ast} 

\usepackage{bm}

\newtheorem{theorem}{Theorem}[section]
\newtheorem{lemma}[theorem]{Lemma}

\newtheorem{proposition}[theorem]{Proposition}
\newtheorem{remark}[theorem]{Remark}

\usepackage{subfigure}  
\usepackage{float}
\usepackage[]{caption}
\definecolor{orange}{rgb}{1,0.5,0}
\definecolor{Red}{rgb}{.795,0.015,0.017}
\definecolor{Ggreen}{rgb}{0.,0.675,0.0128}
\definecolor{Bblue}{rgb}{0.16,.32,0.91}

\usepackage[hyphens]{url}
\usepackage{multicol}

 \usepackage[backref=page]{hyperref}
\hypersetup{
    colorlinks = true,
linkcolor={Red},
urlcolor={blue},
citecolor={Ggreen},    
urlcolor = {blue},
citebordercolor = {0.33 .58 0.33},
 linkbordercolor = {0.99 .28 0.23},
 breaklinks=true
}

\renewcommand*{\backref}[1]{}
\renewcommand*{\backrefalt}[4]{%
  \ifcase #1 %
No citations.
  \or
(page #2).%
  \else
(pages #2).%
  \fi%
}
\usepackage{cite}

\makeatletter
\@namedef{subjclassname@2020}{\textup{2020} Mathematics Subject Classification}
\makeatother
\begin{document}

\title[On denominators of consecutive $\SL$-saturated Farey fractions]{On denominators of consecutive $\SL$-saturated
Farey fractions}
\author{Jack Anderson, Florin P. Boca, Cristian Cobeli, Alexandru Zaharescu}

\date{\today}


\address[Jack Anderson]{Department of Mathematics, University of Illinois at Urbana-Champaign, Urbana, IL 61801, USA.}
\email{jacka4@illinois.edu}

\address[Florin P. Boca]{Department of Mathematics, University of Illinois at Urbana-Champaign, Urbana, IL 61801, USA.}
\email{fboca@illinois.edu}

\address[Cristian Cobeli]{"Simion Stoilow" Institute of Mathematics of the Romanian Academy,~21 Calea Grivitei Street, P. O. Box 1-764, Bucharest 014700, Romania}
\email{cristian.cobeli@imar.ro}

\address[Alexandru Zaharescu]{Department of Mathematics,University of Illinois at Urbana-Champaign, Urbana, IL 61801, USA,
%
and 
"Simion Stoilow" Institute of Mathematics of the Romanian Academy,~21 
Calea Grivitei 
Street, P. O. Box 1-764, Bucharest 014700, Romania}
\email{zaharesc@illinois.edu}

\subjclass[2020]{Primary 11B57.
Secondary: 11J71,   11K36,  11L05.}


\thanks{Key words and phrases: $\operatorname{SL}(2,{\mathbb N})$-saturated
Farey fractions; mediant insertion; consecutive denominators}

\begin{abstract}
The sequence $({\mathscr S}_Q)_Q$ of 
$\SL$-saturated Farey fractions was defined in our previous work by
${\mathscr S}_Q := \{ a/q \in {\mathbb Q} \cap (0,1]: q+a+\bar{a} \le Q\}$, where $\bar{a}$ is the multiplicative
inverse of $a\pmod{q}$ in $[1,q)$.
Here, we prove that the set of $Q$-scaled denominators of consecutive fractions in ${\mathscr S}_Q$ is dense in the region
${\mathcal V}:=\{ (x,y)\in [0,1]^2 : \max \{ (1-3x)/2,2x-1\} \le y \le \max \{ x,1-x\} \}$, and
provide a formula for their distribution in ${\mathcal V}$ as $Q\rightarrow \infty$.
\end{abstract}
\maketitle

\section{Introduction}\label{Sect1}
Motivated by the study of the distribution of reduced quadratic irrationals~\mbox{\cite{Boca2007, KOPS2001, Pol1986, Tec2024, Ust2013}},
we considered in \cite{ABCZ2025} the sequence $(\SF_Q)_Q$ of \emph{$\SL$-saturated Farey fractions},
defined by
\begin{equation*}
\SF_Q := \left\{ \frac{d}{b} : 
\begin{matrix} \text{ there exist } a,c \in \NN 
 \text{ such that }
\left( \begin{smallmatrix} a & b \\ c & d \end{smallmatrix} \right) \in \operatorname{SL}(2,\ZZ)   \\
a\ge b\ge d \ge 1 ,\ a\ge c \ge d,\ a+d \le Q \end{matrix} \right\} ,
\quad  Q\ge 3,
\end{equation*}
viewed as a subset of the customary  Farey set
\begin{equation*}
\cF_Q :=\left\{ \frac{d}{b} : d,b\in \NN, \  1\le d \le  b \le  Q,\   \gcd (d,b)=1 \right\} .
\end{equation*}

As observed in \cite{ABCZ2025}, this set can also be conveniently described as
\begin{equation*}
\SF_Q =\bigg\{ \frac{a}{q} \in \QQ \cap (0,1] : h \bigg( \frac{a}{q}\bigg) :=q+a+\va \le Q\bigg\} ,\quad Q \ge 3,
\end{equation*}
where $\va$ denotes the multiplicative inverse of $a\pmod{q}$ in $[1,q)$.  
In particular, this shows that $\cF_Q \subseteq \SF_{3Q} \subseteq \cF_{3Q}$, and so
\begin{equation}\label{cup}
\bigcup_{Q\ge 3} \SF_Q = \bigcup_{Q\ge 1} \cF_Q = {\mathbb Q} \cap (0,1] .
\end{equation}
Note also that $h(a/q)\neq h(1-a/q)$, but some sort of 
$(x\mapsto 1-x)$-symmetry of $h$ is present in the identity
\begin{equation*}
h\bigg( \frac{a}{q} \bigg) + h \bigg( 1-\frac{a}{q}\bigg) =4q .
\end{equation*}

To compare the sets $\SF_Q$ and $\cF_Q$, look at
\begin{align*}
\cF_7 & = \bigg\{ \frac{1}{7}, \frac{1}{6}, \frac{1}{5}, \frac{1}{4}, \frac{2}{7}, \frac{1}{3},
\frac{2}{5},\frac{3}{7}, \frac{1}{2}, \frac{4}{7}, \frac{3}{5},\frac{2}{3}, 
\frac{5}{7}, \frac{3}{4}, \frac{4}{5}, \frac{5}{6}, \frac{6}{7}, \frac{1}{1}\bigg\} ,  \\
\SF_7 &  = \bigg\{ \frac{1}{5}, \frac{1}{4},\frac{1}{3}, \frac{1}{2}, \frac{2}{3}, \frac{1}{1}\bigg\} , \\
\SF_{10} & =\bigg\{ \frac{1}{8},
\frac{1}{7}, \frac{1}{6}, \frac{1}{5},\frac{1}{4}, \frac{1}{3}, \frac{2}{5},  \frac{1}{2},
\frac{3}{5}, \frac{2}{3}, \frac{3}{4}, \frac{1}{1}\bigg\} \quad \text{and}  \\
\cF_{10} \setminus \SF_{10} & = \bigg\{
 \frac{1}{10}, \frac{1}{9},\frac{2}{9} ,
\frac{2}{7}, \frac{3}{10}, \frac{3}{8}, 
\frac{3}{7}, \frac{4}{9}, 
\frac{5}{9}, \frac{4}{7}, \frac{5}{8} , 
\frac{7}{10}, \frac{5}{7},
\frac{7}{9}, \frac{4}{5}, \frac{5}{6}, \frac{6}{7}, 
\frac{7}{8}, \frac{8}{9},
\frac{9}{10}\bigg\} .
\end{align*}

Besides property \eqref{cup}, the sets $\SF_Q$ exhibit another interesting arithmetic property resembling $\cF_Q$.  As shown in the appendix to \cite{ABCZ2025},
the elements of each set $\SF_Q^* :=\{ 0\} \cup \SF_Q$ define a {\it unimodular partition} of the interval $[0,1]$.
This means that if $0/1 < a_1/q_1= 1/(Q-2) < \cdots < a_{\# \SF_Q}/b_{\# \SF_Q}=1/1$ denote the elements 
of $\SF^*_Q$, then $a_{i+1} q_i -a_i  q_{i+1} =1$ for any $i=1,\ldots,\# \SF_Q -1$.

However, in sharp contrast with $(\cF_Q)_Q$, the elements of $(\SF_Q)_Q$ are no longer uniformly
distributed as $Q\rightarrow \infty$. Furthermore, it was proved in \cite[Theorem~1]{ABCZ2025} that
\begin{equation}\label{DQasym}
\# \big( \SF_Q \cap [0,\beta]\big) = \frac{Q^2}{2\zeta(2)} \log \bigg( \frac{2+2\beta}{2+\beta}\bigg)
+O_\varepsilon (Q^{3/2+\varepsilon}) ,\quad \forall\beta \in [0,1].
\end{equation}
In particular,  we have
\begin{equation}\label{assymetry}
\lim\limits_Q \frac{\# (\SF_Q \cap [0,1/2])}{\# \SF _Q} =\frac{\log (6/5)}{\log (4/3)} \approx 0.63376  .
\end{equation}

\begin{figure}[thb]
\centering
\hfill
\includegraphics[width=0.58\textwidth]{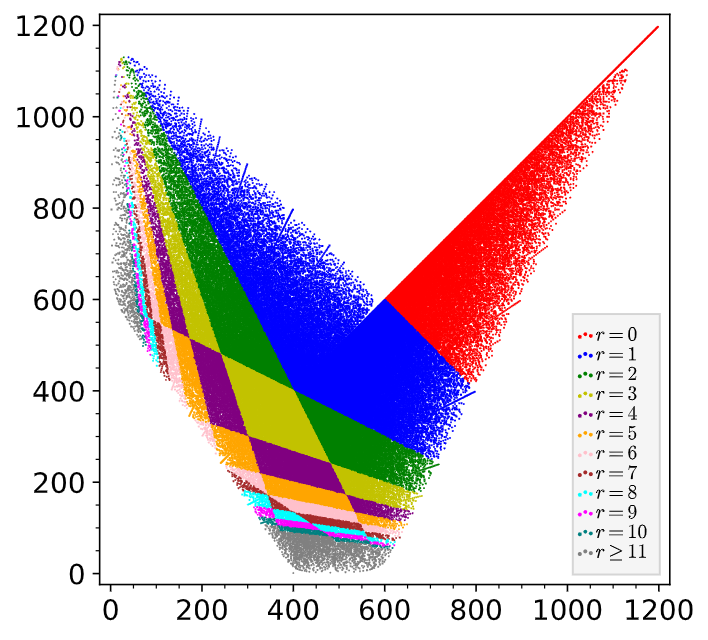} 
\hfill\mbox{}
\caption{The set of pairs of denominators of consecutive fractions
in $\SF_{1200}$. 
The points are colored differently, depending on the number of intermediate Farey fractions whose insertion is delayed at this level.
}
\label{FigureDenominatorsApart1200}
\end{figure}

It is well-known that each set  \mbox{$\cF^*_Q :=\{ 0\} \cup \cF_Q$} defines a unimodular partition of the interval $[0,1]$.
In particular, the elements of~$\cF^*_Q$ are determined by their denominators only.
The elements of the~set
\begin{equation*} 
\cD_Q :=\big\{ (q_1,q_2)\in \NN^2: \exists\, a_1,a_2 \in \ZZ,\ a_1/q_1 <a_2/q_2 
\text{ consecutive elements in } \cF^*_Q \big\} 
\end{equation*}
are in one-to-one correspondence with the~set of primitive lattice points in the triangle $Q\cT$,
where~$\cT$ denotes the Farey triangle $\{ (x,y)\in [0,1]^2: x+y >1\}$. As a result, the scaled set~$\cD:=\cup_Q  Q^{-1} \cD_Q$ is dense in the closed triangle $\overline{\cT}$.
It was first noticed by Kargaev and Zhigljavsky~\cite{KZ1997} (see also \cite[Lemma~2]{BCZ2001} for
a quick effective proof) that the $Q$-scaled sets $Q^{-1} \cD_Q$
are uniformly distributed in $\overline{\cT}$, in the sense that, for every 
rectangular box $\cR\subseteq \overline{\cT}$,
\begin{equation}\label{eq2}
\lim\limits_Q \frac{\# (\cD_Q \cap Q\cR)}{\# \cD_Q}=2\operatorname{Area} (\cR).
\end{equation}
As noticed in \cite{BCZ2001}, the shifting of denominators of consecutive elements in $\cF_Q^*$  is conveniently described by 
iterating the map $T:\cT \rightarrow \cT$,
\begin{equation*}
T(x,y) :=\big( y, \kappa (x,y) y-x\big) \quad \text{with }
\kappa (x,y):=\bigg\lfloor \frac{1+x}{y}\bigg\rfloor ,
\end{equation*}
which is invertible, area-preserving, and satisfies, for every $Q\ge 1$,
\begin{equation*}
T\bigg( \frac{q_1}{Q},\frac{q_2}{Q}\bigg) =\bigg( \frac{q_2}{Q},\frac{q_3}{Q}\bigg),
\end{equation*}
if $a_1/q_1 < a_2/q_2 < a_3/q_3$ are consecutive elements in $\cF^*_Q$.

The primary aim of this paper  is to investigate the existence of an analogue of~\eqref{eq2}
when~$\cF_Q$ is being replaced by $\SF_Q$. For every $Q\ge 3$, we consider the set
$\mathscr{D}_Q$ of denominators $(q_1,q_2)$ of consecutive elements
$a_1/q_1 <a_2/q_2$ from $\SF_Q$. 
The elements of the set $\mathscr{D}_{1200}$ are plotted in Figure~\ref{FigureDenominatorsApart1200}.
The role of $\overline{\cT}$ will be played here by the region
\begin{equation*} 
\cV:=\bigg\{ (x,y)\in [0,1]^2: \max \bigg\{ \frac{1-3x}{2} , 2x-1 \bigg\} \le y \le \max \{ 1-x,x\} \bigg\}.
\end{equation*}

We  prove the following result:

\begin{theorem}\label{Thm1}
The set $\mathscr{D}:=\cup_Q Q^{-1} \mathscr{D}_Q$ is dense in $\cV$.
\end{theorem}

We will first show the inclusion $\scrD \subseteq \cV$, then 
prove Theorem \ref{Thm2}, a stronger result which provides an asymptotic 
formula for the cardinality of  the set $\scrD_Q \cap Q\cR$ for rectangular boxes
$\cR$ inside $\cV$. This will show  in particular that $Q^{-1} \scrD_Q \cap \cR \neq \emptyset$ 
for large enough $Q$ (depending on $\cR$). 

\begin{theorem}\label{Thm2}
The $Q$-scaled sets $Q^{-1} \mathscr{D}_Q$ have an asymptotic limiting distribution, described
as follows. Partition the set $\cV$ as $\cV_1 \cup \cV_2 \cup \cV_3$, with\footnote{See the left-sidde of Figure \ref{FigureVandWs}.}
\begin{equation}\label{eq3}
\begin{split}
& \cV_1 :=\bigg\{ (x,y) \in \bigg[ \frac{1}{2},1\bigg] \times 
\bigg[\frac{1}{3},1\bigg]: \max \{1-x,2x-1\} \le y \le x\bigg\} , \\
& \cV_2 := \bigg\{ (x,y) \in \bigg[ \frac{1}{5},\frac{2}{3} \bigg] \times \bigg[ 0,\frac{1}{2} \bigg] :
\max \bigg\{ \frac{1-3x}{2} ,0,2x-1\bigg\} \le y \le \min\{ x,1-x\} \bigg\} , \\
& \cV_3 := \bigg\{ (x,y)\in \bigg[0,\frac{1}{2} \bigg] \times \bigg[\frac{1}{5},1\bigg] :  
 \max \bigg\{ \frac{1-3x}{2}, x\bigg\} \le y \le 1-x\bigg\} .
\end{split}
\end{equation}
Then, for every rectangular box $\cR \subseteq \cV _i$, $i=1,2,3$,
\begin{equation*}
\lim\limits_Q \frac{\#( \mathscr{D}_Q \cap Q\cR)}{\# \mathscr{D}_Q} =\frac{2}{\log (4/3)} \iint_{\cR} H_i (x,y)\, dx dy,
\end{equation*}
with $H_i (x,y) > 0$ on $\mathring{\cV}_i$ given by \eqref{H1}, \eqref{H2} and \eqref{H3}.
\end{theorem}

\begin{figure}[thb]
\centering
\hfill
\includegraphics[width=0.408\textwidth]{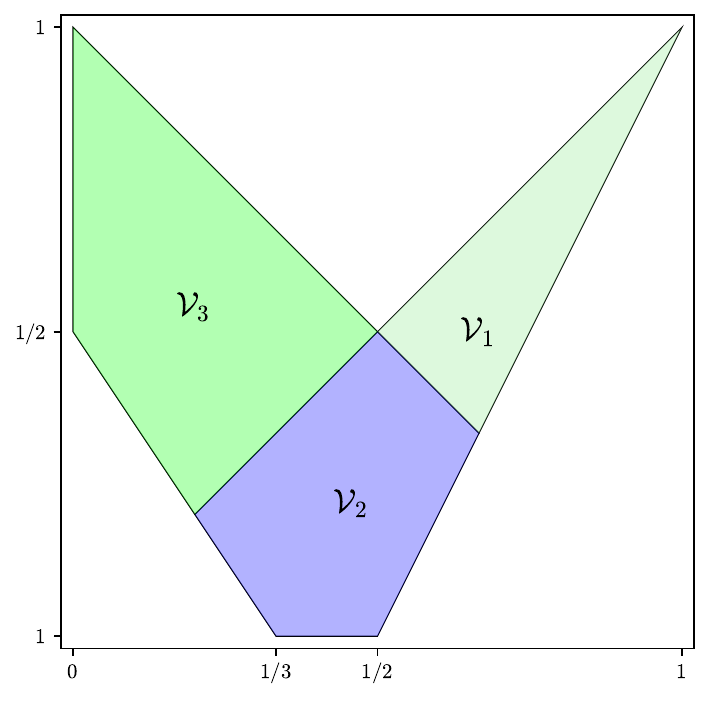} 
\includegraphics[width=0.408\textwidth]{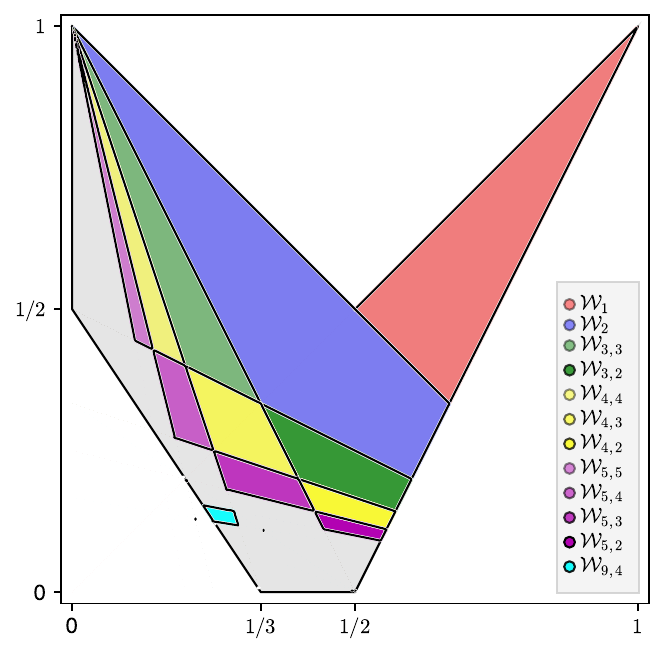} 
\hfill\mbox{}
\caption{The polygons defined in~\eqref{eq3} and the partition 
of the $\cV$-shape described in Section~\ref{partition}. 
The order in which the polygons are positioned is from right 
to left and from top to bottom.
}
\label{FigureVandWs}
\end{figure}

The region $\cV$ can be partitioned into subregions $\cW_1,\cW_2,\cW_3,\ldots$, with $\cW_r$ representing 
the closure of the scaled set $\cup_Q Q^{-1} \mathscr{D}_Q^{(r)}$, 
where $\mathscr{D}_Q^{(r)}$  denotes the
set of pairs of denominators of  consecutive elements of $\SF_Q$ with exactly $r-1$ elements of $\cF_Q$ in between (see the representation on
the right-hand side of  Figure~\ref{FigureVandWs}).
In addition, every $\cW_r$, $r\ge 3$, can be
further partitioned into $r-1$ (possibly  empty) polygons $\cW_{r,i}$, with $i\in \{2,\ldots,r\}$ being the (unique)
place where the mediant of consecutive denominators is being attained, that is
$q_i =q_1 +q_{r+1}$. The boundaries of the polygons $\cW_{r,i}$ are computable, as illustrated in
Section~\ref{partition} for $r\le 5$. See also Figures~\ref{FigureDenominatorsApart1200} and \ref{FigureVandWs}.

If the non-negative rational numbers $a_1/q_1 < a_2/q_2 < a_3/q_3$ satisfy $a_2 q_1 -a_1 q_2 =1
=a_3 q_2 -a_2 q_3$, then
\begin{equation*}
\frac{q_1+q_3}{q_2} =\frac{a_1+a_3}{a_2} =a_3 q_1 -a_1 q_3 \in \NN.
\end{equation*}
Furthermore, when $a_1/q_1 < a_2/q_2 < a_3/q_3$ are consecutive in $\cF_Q^*$, we have
\begin{equation*}
\frac{q_1+q_3}{q_2} =\bigg\lfloor \frac{Q+q_1}{q_2}\bigg\rfloor =\kappa \bigg( \frac{q_1}{Q},\frac{q_2}{Q}\bigg)
=: \nu_Q \bigg( \frac{a_2}{q_2}\bigg).
\end{equation*}
This quantity, called the \emph{index} of the fraction $a_2/q_2$ (in $\cF_Q$), plays a critical role in the distribution
of the Farey sequence $(\cF_Q)_Q$ and has been investigated by several authors
(see, e.g., \cite{BCZ2001,HS2003,BGZ2002,ALZ2007,ALVZ2008,CCG2024,
Kor2025a,Kor2025b,Li2023-1,Li2023-2}).
It would be interesting to find an analogue formula for the index $(q_1+q_3)/q_2$ when $a_1/q_1 <
a_2/q_2 < a_3/q_3$ are consecutive elements in $\SF_Q$.

One nice feature of the index, proved in \cite[Theorem 1]{HS2003} through an elementary mediant 
argument, is the identity
\begin{equation}\label{index}
\sum\limits_{i=1}^{R_Q} \nu_Q (f_i) =3R_Q -1,
\end{equation}
where $R_Q :=\# \cF_Q =\varphi(1)+\varphi(2)+\cdots +\varphi(Q)$, 
$f_1 =1/Q<\cdots < f_{R_Q}=1/1$
denote the elements of $\cF_Q$, $f_0:=0/1$ and $f_{1+R(Q)}:=1+f_1$.
In Section~\ref{Sect5} we show that a similar identity holds when $\cF_Q$ is replaced by $\SF_Q$.
Setting $S_Q :=\# \SF_Q$, we enumerate the elements of $\SF_Q$ as
$\gamma_1 =1/(Q-2) <\gamma_2 < \cdots < \gamma_{S_Q}=1/1$, and take
$\gamma_0 :=0/1$, $\gamma_{1+S_Q}:=1+\gamma_1$, $\gamma_i=a_i/q_i$. A quick argument as in the proof of~\eqref{index} 
provides the following statement.

\begin{proposition}\label{Prop3}
For every $Q\ge 3$, we have
\begin{equation*}
\sum\limits_{i=1}^{S_Q} \frac{q_{i-1}+q_{i+1}}{q_i} =3S_Q -1.
\end{equation*}
\end{proposition}

\begin{figure}[thb]
\centering
\hfill
\includegraphics[width=0.48\textwidth]{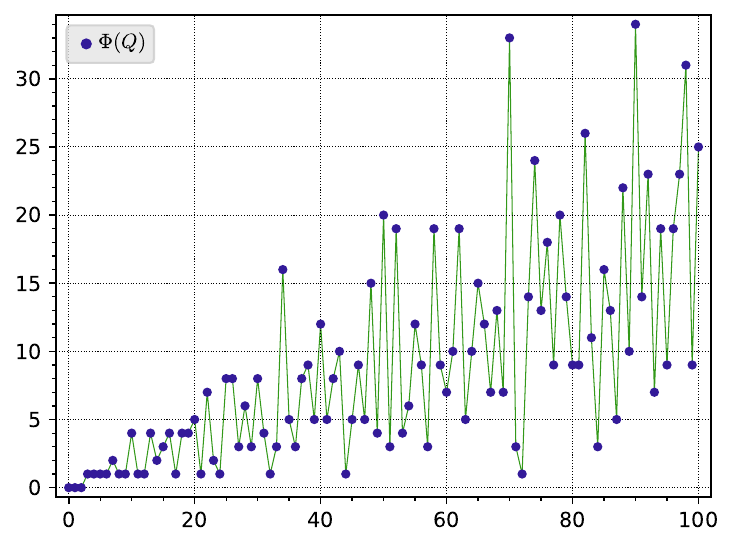}
\includegraphics[width=0.48\textwidth]{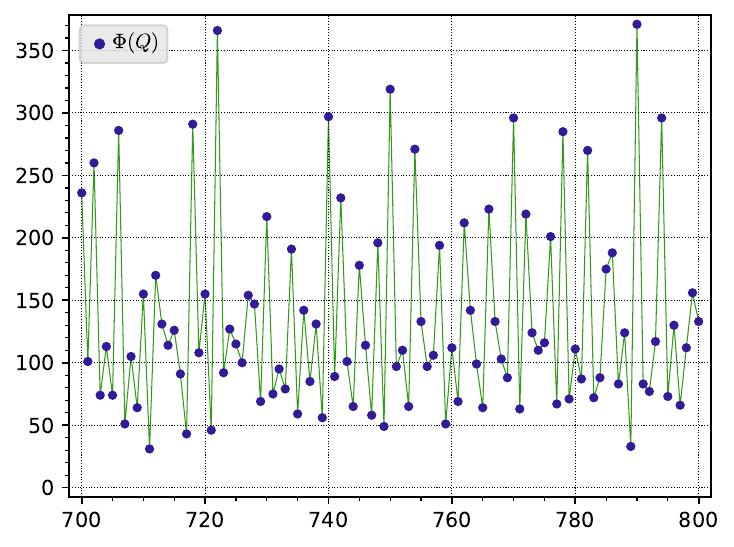}
\hfill\mbox{}
\caption{The highly oscillatory behavior of $\Phi(Q)$.
}
\label{FigureGraphOfPhi}
\end{figure}

Here, the function 
\begin{equation*}
\Phi(Q):=S_Q -S_{Q-1} =\# \bigg\{ \frac{a}{q} \in \QQ \cap (0,1] : q+a+\va =Q\bigg\}
\end{equation*}
arises as a substitute for the
Euler function $\varphi (Q)=R_Q -R_{Q-1}$, measuring the number of mediant insertions from
level $Q-1$ to level $Q$. 
For the correlated study of pairs 
$(a/q,\va/q)$ see~\cite{Hum2022, Shp2012, BRS2024} and the references therein.
Nevertheless, the arithmetic function  $\Phi$ is not multiplicative and exhibits an erratic behavior (see Figure~\ref{FigureGraphOfPhi}).

The asymmetry of the fractions in $\SF _Q$, noticed in \eqref{assymetry},
does not imply that in the case of symmetric fractions the one on the left will always appear before the one on the right by insertion.
For example, $5/9\in \SF _{16}$, but $4/9$ first appears in $\SF _{20}$.
On a different  note, we remark that the circle packing generated by $\SF _Q$ 
belongs to the class of those produced by various rules that delay the insertion
(see the examples in~\cite{OH2000, GM2010}), 
but ultimately, even though the partial configurations are different, 
the integer packing produced by $\SF _Q$ will coincide with 
that produced by $\cF _Q$.

\section{\texorpdfstring{The inclusion $\mathscr{D}\subseteq \cV$}{The inclusion scrD subseteq cV}}\label{Sect2}
Let $r\ge 1$. Suppose $a_1/q_1 <\cdots <a_{r+1}/q_{r+1}$ are consecutive elements in $\cF_Q$, $a_1/q_1 ,a_{r+1}/q_{r+1}\in \SF_Q$ and
$a_2/q_2,\ldots ,a_r/q_r \notin \SF_Q$ (in other words $a_1/q_1 <a_{r+1}/q_{r+1}$ are consecutive elements in $\SF_Q$).
 Denote by $\mathscr{D}_Q^{(r)}$ the set of denominator pairs $(q_1,q_{r+1})$ of consecutive elements of $\SF_Q$, with $q_1$ and $q_{r+1}$ as above.

We start with the easiest case $r=1$.

\begin{lemma}\label{L4}
$\mathscr{D}_Q^{(1)} \subseteq  Q \cV_1$.
\end{lemma}

\begin{proof}
As in Section 3.1 of \cite{ABCZ2025}, in this case we have $q_2<q_1$, so $q_1 \in [Q/2,Q]$, and also
$3q_1 -Q-q_2 \le\vq_2 <q_1$ with $\vq_2$ denoting the multiplicative inverse of $q_2 \pmod{q_1}$ in $[1,q_1)$,
 leading to $q_2 >2q_1-Q$. We also employ the trivial inequality $q_2 >Q-q_1$.
\end{proof}

The situation $r\ge 2$ requires additional work.
The following critical equality was proved in the appendix to \cite{ABCZ2025}:
\begin{equation}\label{unimodular}
a_{r+1} q_1 - a_1 q_{r+1}=1 .
\end{equation}

We consider the index of $a_i/q_i$ in $\cF_Q$, given by
\begin{equation*}
\nu_i:=\kappa \bigg( \frac{q_{i-1}}{Q}, \frac{q_i}{Q}\bigg) =\bigg\lfloor \frac{Q+q_{i-1}}{q_{i}}\bigg\rfloor,
\end{equation*}
so that $q_{i+1}=\nu_i q_i -q_{i-1}$, or equivalently
\begin{equation*}
(q_{i+1},q_i)=(q_i,q_{i-1}) \left( \begin{matrix} \nu_i & 1 \\  -1 & 0 \end{matrix}\right), \quad  i=2,\ldots ,r.
\end{equation*}
The noncommutative polynomials defined by $K^F_{-1}(\cdot) =0$, $K^F_0 (\cdot)=1$, 
\mbox{$K^F_1 (x)=x$}, and
\begin{equation*}
K^F_\ell (x_1,\ldots,x_\ell) =x_\ell K^F_{\ell -1} (x_1,\ldots,x_{\ell-1}) -K^F_{\ell-2} (x_1,\ldots,x_{\ell-2}) ,
\end{equation*}
satisfy
\begin{equation}\label{matrix}
 \left( \begin{matrix} K^F_{i-1} (\nu_2,\ldots,\nu_i )  &  K^F_{i-2} (\nu_2 ,\ldots, \nu_{i-1}) \\  
- K^F_{i-2} (\nu_3,\ldots,\nu_i) &  - K^F_{i-3} (\nu_3,\ldots,\nu_{i-1})\end{matrix}\right)  
= \left( \begin{matrix} \nu_2 & 1 \\ -1 & 0 \end{matrix}\right) \cdots 
\left( \begin{matrix} \nu_i & 1 \\ -1 & 0 \end{matrix} \right),\quad i\ge 2,
\end{equation}
leading to 
\begin{equation*}
(q_{i+1},q_i) = (q_2,q_1) \left( \begin{matrix} K^F_{i-1} (\nu_2,\ldots,\nu_i )  &  K^F_{i-2} (\nu_2 ,\ldots, \nu_{i-1}) \\  
- K^F_{i-2} (\nu_3,\ldots,\nu_i) &  - K^F_{i-3} (\nu_3,\ldots,\nu_{i-1})\end{matrix}\right),\quad  i=2,\ldots,r.
\end{equation*}
Since $a_{i+1} =\nu_i a_i -a_{i-1}$, we also have
\begin{equation*}
(a_{i+1},a_i) = (a_2,a_1) \left( \begin{matrix} K^F_{i-1} (\nu_2,\ldots,\nu_i )  &  K^F_{i-2} (\nu_2 ,\ldots, \nu_{i-1}) \\  
- K^F_{i-2} (\nu_3,\ldots,\nu_i) &  - K^F_{i-3} (\nu_3,\ldots,\nu_{i-1})\end{matrix}\right),\quad  i=2,\ldots,r.
\end{equation*}

\begin{lemma}\label{L5}
$K^F_{r-1} (\nu_2,\ldots,\nu_{r})=1.$
\end{lemma}

\begin{proof}
Denoting by $\left( \begin{smallmatrix} K^\prime & \ast \\  -K^{\prime\prime}  & \ast \end{smallmatrix}\right)$   the matrix given by \eqref{matrix} for $i=r$, we have 
\begin{equation*}
q_{r+1}=K^\prime q_2 -K^{\prime\prime} q_1 \quad \text{and}\quad 
a_{r+1}=K^\prime a_2 -K^{\prime\prime} a_1.
\end{equation*}
Taking stock of \eqref{unimodular} and using the fact that $a_1/q_1 <a_2/q_2$ are consecutive elements in $\cF_Q$, so $a_2 q_1-a_1 q_2 =1$, we find
\begin{equation*}
1=a_{r+1} q_1 -a_1 q_{r+1} =K^\prime (a_2 q_1 -a_1 q_2)=K^\prime ,
\end{equation*}
as desired.

\end{proof}

\begin{lemma}\label{L6}
For every $r\geq 2$, we have $q_1+q_{r+1} \le Q$.
\end{lemma}

\begin{proof}
Suppose, by contradiction, that $q_1+q_{r+1} >Q$. Equality \eqref{unimodular} yields 
$\gcd (q_1,q_{r+1})=1$. Using the inequality $q_1 +q_{r+1}>Q$ and the bijection between pairs of denominators of consecutive 
elements in $\cF_Q$ and primitive lattice points $(q,q^\prime)$ in the triangle $Q\cT$, it follows that 
$q_1$ and $q_{r+1}$ are denominators of two consecutive elements in $\cF_Q$. But this is a contradiction because $r\ge 2$ and
$a_1 /q_1 <a_2/q_2< a_{r+1}/q_{r+1}$.
\end{proof}

\begin{lemma}\label{L7}
For every $r\ge 2$, we have $3q_1+2q_{r+1} >Q$.
\end{lemma}

\begin{proof}
As $a_1/q_1 < \cdots < a_{r+1}/q_{r+1}$ are consecutive in $\cF_Q$,  Lemma~\ref{L6} yields $q_1+q_{r+1} =q_i$ for some $i\in \{ 2,\ldots,r\}$. We have
\begin{equation*}
(q_{i+1} ,q_i) =(q_2,q_1) \left( \begin{matrix} a & b \\ c & d \end{matrix}\right) =
(q_{r+1} ,q_r) \left( \begin{matrix} \alpha & \beta \\ \gamma & \delta \end{matrix} \right)
\end{equation*}
for some integers $a,b,c,d,\alpha,\beta,\gamma,\delta$ with $ad-bc =1=\alpha \delta -\beta \gamma$. 
Lemma \ref{L5} yields
\begin{equation}\label{K-nu}
(q_{r+1},q_r) =(q_2,q_1) \left( \begin{matrix} 1 & \nu \\ -K & 1-K \nu \end{matrix}\right)
\end{equation}
for some (positive) integers $K,\nu$, showing that $q_r=q_1+\nu q_{r+1}$. Hence, we gather
\begin{equation*}
\begin{cases}
 q_{i+1} = \alpha q_{r+1} +\gamma q_r = \alpha q_{r+1} +\gamma (q_1+\nu q_{r+1}) = \gamma q_1 +(\alpha+\nu \gamma) q_{r+1} \\ 
q_i = \beta q_{r+1}+\delta q_r = \beta q_{r+1} +\delta (q_1 +\nu q_{r+1}) =\delta q_1 +(\beta +\nu \delta) q_{r+1},
\end{cases}
\end{equation*}
which leads to 
\begin{equation*}
(q_{i+1},q_i) =(q_1 ,q_{r+1}) \left( \begin{matrix} A & B \\ C & D \end{matrix}\right) ,
\end{equation*}
with $A=\gamma$, $B=\delta$,  $C=\alpha +\nu \gamma$, $D=\beta+\nu\delta$ satisfying 
$AD-BC=\beta \gamma -\alpha \delta =-1$. Since $q_i =q_1+q_{r+1}$, we find $B=D=1$ and so $C=A+1$
and $q_{i+1}=Aq_1 +(A+1)q_{r+1}$ with $A\ge 0$. This yields $\lfloor q_{i+1}/q_i\rfloor = \lfloor (Aq_1+(A+1)q_{r+1})/(q_1+q_{r+1})\rfloor=A$.

Applying one of the inequalities in \cite[Eq.~(16)]{ABCZ2025} we infer
\begin{equation*}
\begin{split}
Q & < (2+\lfloor q_{i+1}/q_i\rfloor ) q_i -q_{i+1} +a_i \\  & \le (3+A) q_i -q_{i+1}
=(3+A)(q_1+q_{r+1})-\big( Aq_1 +(A+1)q_{r+1}\big) =3q_1+2q_{r+1} ,
\end{split}
\end{equation*}
as desired.
\end{proof}

\begin{proposition}\label{P8}
For every $r\ge 2$, we have
\begin{equation*}
\mathscr{D}_Q^{(r)} \subseteq Q\cV_2 \cup Q\cV_3 .
\end{equation*}
\end{proposition}

\begin{proof}
Lemmas~\ref{L6} and \ref{L7} provide $q_1+q_{r+1} \le  Q$ and $3q_1+2q_{r+1} \ge Q$, so it remains to prove the inequality
$q_{r+1} -2q_1 +Q \ge 0$.

With $K$ as in equality \eqref{K-nu}, the last inequality  is obsolete when $K\ge 2$. Indeed, 
it was observed in Sections 3.2 and 3.3 of \cite{ABCZ2025} that when $r\ge 2$ we have $q_1 <q_2$,
and we can use
$q_{r+1}=q_2 -Kq_1 > 0$ to infer $q_1 \le q_2/K \le Q/2$. When $K=1$, we have $q_{r+1}-2q_1=q_2 -3q_1$. Now, the first inequality in \cite[Eq.~(16)]{ABCZ2025},
where $q=q_2 > \alpha=q_1$,  yields $3q_1 -q_2 \le ( 2+[q_2/q_1]) q_1-q_2+a_1 \le Q$, which leads to $q_2-3q_1+Q \ge 0$. 
Finally, we notice that $K\le 0$ cannot occur, as $K\le -1$ would imply $Q\ge q_{r+1} = q_2 -Kq_1 \ge q_2+q_1 >Q$, a contradiction,
while $K=0$  yields $q_{r+1}=q_2$ and $q_r=q_1$, a contradiction as $r>1$.  
\end{proof}

\section{\texorpdfstring{The inclusion $\cV \subseteq \overline{\mathscr{D}}$ and the proof of Theorem \ref{Thm2}}
{The inclusion cV subseteq scrD}}\label{Sect3}
We start off the proof of Theorem \ref{Thm2} 
by fixing $(x_0,y_0)\in \cR$ and $\varepsilon_1,\varepsilon_2>0$ such that 
 \begin{equation*}
\cR:=[x_0-\varepsilon_1,x_0+\varepsilon_1] \times [y_0-\varepsilon_2,y_0+\varepsilon_2] =: I_1 \times I_2 \subseteq
\mathring{\cV}_1 \cup \mathring{\cV}_2  \cup \mathring{\cV}_3.
\end{equation*}
We will count the number of integer pairs $(q_1,q_2) \in \scrD_Q \cap Q\cR $.
The situation where $\cR \subseteq \mathring{\cV}_1$, so  $x_0+y_0 >1$, $x_0 > y_0$  and $1+y_0-2x_0 >0$, is considered first. 
Here, we can look for $q_1,q_2$ with $q_1+q_2 >Q>2q_1 -q_2$ and $q_1 >q_2$. Hence, they are   denominators of consecutive elements in $\cF_Q$ and the 
considerations from Section 3.1 of \cite{ABCZ2025} (see especially Corollary 6 thereby) yield

\begin{itemize}
\item[(i)]
$h(a_1/q_1)=q_1+a_1+\va_1 =3q_1-q_2 -\vq_2$, so $a_1/q_1 \in \SF_Q \Longleftrightarrow 3q_1 -q_2 -\vq_2 \le Q$.
\item[(ii)]
$q_2 <q_1$, so $h(a_1/q_1) \le Q \Longrightarrow h(a_2/q_2)=q_2 +a_2 +\vva_2 \le Q$,
\end{itemize}
where $\bar{n}$ denotes the multiplicative inverse of $n\pmod{q_1}$ in $[1,q_1)$ and $\bar{\bar{n}}$ denotes the multiplicative
inverse of $n\pmod{q_2}$ in $[1,q_2)$.

This shows that 
we need to count the integer triples $(q_1,q_2,b)$ 
such that\footnote{Here, $b$ plays the role of $\vq_2 =q_1 -a_1$. Note also that $3q_1-q_2-Q > 2q_1-Q >0$.} 
\begin{equation}\label{V1}
q_1 \in QI_1,\quad q_2\in QI_2, \quad q_ 2 b \equiv 1 \pmod{q_1}, \quad
3q_1 -q_2 -Q \le  b <q_1,
\end{equation}
with inequalities $ \max\{ Q-q_1,2q_1-Q\} < q_2 < q_1$ following from
\mbox{$\cR =I_1 \times I_2 \subseteq \mathring{\cV}_1$}.

For $u \in QI_1$, consider the set
\begin{equation*}
\Omega_{Q,u}^{(1)}:= \{ (v,w): v\in QI_2, F_Q (u,v)  :=3u -v-Q \le w \le G_Q(u,v) :=u\}.
\end{equation*}
We apply \cite[Lemma~3.1]{Boca2007} to each set $\Omega^{(1)}_{Q,q_1}$, $q_1 \in QI_1$, taking $T=\lfloor Q^{1/4}\rfloor$. 
Since $0\le v\le u$ on $Q\mathring{\cV}_1$, the  range $[\alpha ,\beta]$ of $v$ has length $\beta-\alpha \le q_1$
and the total variations of $F_Q(q_1,\cdot)$ and $G_Q(q_1,\cdot)$ satisfy $V_\alpha^\beta F_Q(q_1,\cdot) \le q_1$,
$V_\alpha^\beta G_Q(q_1,\cdot)=0$. We also have $\| F_Q (q_1,\cdot)\|_\infty \le \| G_Q (q_1,\cdot)\|_\infty \le q_1$. In the end, we infer that 
the number of integer triples $(q_1,q_2,b)$ that satisfy \eqref{V1} is given by
\begin{equation*}
S_1(Q)= \sum\limits_{q_1 \in QI_1}  \sum\limits_{\substack{q_2 \in QI_2 \\ 3q_1 - q_2 -Q \le b <q_1 \\  q_2 b \equiv 1 \pmod{q_1}}} 1  =
\sum\limits_{q_1 \in QI_1} \bigg( \frac{\varphi(q_1)}{q_1^2}\,\operatorname{Area}(\Omega_{Q,q_1}^{(1)})+O_\varepsilon (\cE_{q_1,\varepsilon}) \bigg),
\end{equation*}
with error terms summing to
\begin{equation*}
\sum\limits_{q_1 \in QI_1} \cE_{q_1,\varepsilon} \ll \sum\limits_{q_1 \in QI_1} \bigg( \frac{q_1}{Q^{1/4} q_1} \, q_1 +Q^{1/4}q_1^{1/2+\varepsilon}\bigg)  
\ll Q^{7/4+\varepsilon}.
\end{equation*}

The inequalities $0< Q+v-2u <u$, $(u,v) \in Q\cR  \subseteq Q\mathring{\cV}_1$, yield
\begin{equation*}
\operatorname{Area}(\Omega_{Q,u}^{(1)}) = \int_{QI_2} (Q+v-2u)\, dv \leq Qu .
\end{equation*}
Applying now \cite[Lemma~2.3]{BCZ2000} to the function 
\begin{equation*}
V(u)=\frac{1}{u} \, \operatorname{Area}(\Omega_{Q,u}^{(1)})
=\int_{QI_2} \frac{Q+v-2u}{u}\, dv,\quad u\in QI_1,
\end{equation*}
which is monotonically decreasing and $C^1$ on the interval $QI_1$, with $\| V\|_\infty \le Q$, we find
\begin{align*}
 S_1(Q) & =\frac{1}{\zeta (2)} \int_{QI_1} V(u)\, du  + O_\varepsilon (Q^{7/4+\varepsilon}) \\ & 
= \frac{1}{\zeta(2)} \int_{QI_1} du \int_{QI_2} dv \, \frac{Q+v-2u}{u}+O_\varepsilon (Q^{7/4+\varepsilon})  \\ &
=    \frac{Q^2}{\zeta(2)} \iint_{\cR} H_1 (x,y) \, dx dy,
\end{align*}
with 
\begin{equation}\label{H1}
H_1(x,y) =\frac{1+y-2x}{x} >0 \quad \text{ on $\mathring{\cV}_1$}.
\end{equation}
In conjunction with \eqref{DQasym}, this allows us to conclude the proof in this case.

Next, we look at the case where $x_0+y_0<1$. Although analogous to the previous one, this is slightly more involved
since there will be one or possibly more fractions in $\cF_Q$ between $a_1/q_1$ and $a_2/q_2$. 
In order to deal with these intermediate fractions, we use the following result:
\begin{lemma}[Lemma~A.1 in \cite{ABCZ2025}]\label{L10}
    Let $a_1/q_1 < a_2/q_2$, $0 \leq a_i < q_i$, with $a_2 q_1 - a_1 q_2 = 1$. Set $Q_i := h(a_i/q_i)$ and $Q_0 := \max\{Q_1, Q_2\}$. 
Consider the mediant $a_\mystar/q_\mystar := (a_1 + a_2)/(q_1 + q_2)$ and set $Q_\mystar:= h(a_\mystar/q_\mystar)$. Then
    \begin{enumerate}
        \item[$(1)$]   $Q_0 < Q_\mystar$.
        \item[$(2)$]   If we also assume that $a_1/q_1 < a_2/q_2$ are consecutive elements in $\SF_{Q_0}$, then $a_1/q_1 < a_2/q_2$ are also consecutive elements in 
$\SF_Q$ if $Q_0 \leq Q < Q_\mystar$, while $a_1/q_1 < a_\mystar/q_\mystar < a_2/q_2$ are consecutive in $\SF_{Q_\mystar}$.
    \end{enumerate}
\end{lemma}

Therefore, 
\begin{equation}\label{SQcharacterization}
\begin{split}
a_1/q_1 <a_2/q_2 \text{ are consecutive in $\SF_Q$ } &
\Longleftrightarrow \  a_1/q_1,a_2/q_2 \in \SF_Q \text{ and } a_\mystar/q_\mystar \notin \SF_Q \\
& \Longleftrightarrow \ \max\{ h(a_1/q_1),h(a_2/q_2)\} \le Q < h(a_\mystar/q_\mystar) .
\end{split}
\end{equation}

Recall also the following:
\begin{lemma}[Lemma 5 in \cite{ABCZ2025}]\label{L11}
Suppose that $a_1/q_1<a_2/q_2$, $0<a_i<q_i$,  and 
$a_2 q_1 -a_1 q_2 =1$.
 Denote by $\bar{n}$ the multiplicative inverse of $n \pmod{q_1}$ in~$[1,q_1)$ and by $\bar{\bar{n}}$ the multiplicative inverse of 
$n\pmod{q_2}$ in~$[1,q_2)$. Then 
\begin{itemize}
 \item[$\rm (i)$]
 $\displaystyle \va_1 =\bigg( 1+\bigg\lfloor \frac{q_2}{q_1}\bigg\rfloor \bigg) q_1 -q_2 ,
 \quad \vva_2 =q_1 -\bigg\lfloor \frac{q_1}{q_2}\bigg\rfloor q_2 ,$

\medskip
 \item[$\rm (ii)$] $\displaystyle a_1 =q_1 -\vq_2,\quad 
 a_2 =\vvq_1 ,\quad a_1 =\frac{q_1 \vvq_1 -1}{q_2} ,
 \quad a_2 =q_2 -\frac{q_2 \vq_2 -1}{q_1}.$

\medskip
 \item[$\rm (iii)$]
$\displaystyle h(a_1/q_1)  =\bigg( 2+\bigg\lfloor \frac{q_2}{q_1}\bigg\rfloor \bigg) 
q_1 -q_2+\frac{q_1 \vvq_1 -1}{q_2} ,$

\medskip
\item[$\rm (iv)$] 
$\displaystyle h(a_2/q_2)= q_1 +
        \bigg( 2-\bigg\lfloor \frac{q_1}{q_2}\bigg\rfloor \bigg) q_2 
        -\frac{q_2 \vq_2 -1}{q_1} .$
\end{itemize}
\end{lemma}

The following expressions of $h(a/q)$ will be helpful:
\begin{lemma}\label{L12}
Under the assumptions of the previous lemma, we have

\begin{itemize}
\item[$\rm (i)$]
$\displaystyle h(a_1/q_1)= \bigg( 3+ \bigg\lfloor  \frac{q_2}{q_1} \bigg\rfloor\bigg) q_1 -q_2 -\vq_2  =\bigg( 2+ \bigg\lfloor \frac{q_2}{q_1}\bigg\rfloor\bigg) q_1 -q_2 
+  \frac{q_1 \vvq_1 -1}{q_2}$.

\medskip

\item[$\rm (ii)$]
$\displaystyle h(a_2/q_2)= \bigg( 1-\bigg\lfloor \frac{q_1}{q_2} \bigg\rfloor \bigg)  q_2+q_1 +\vvq_1 
=q_1 +\bigg( 2- \bigg\lfloor \frac{q_1}{q_2} \bigg\rfloor \bigg) q_2 -\frac{q_2\vq_2 -1}{q_1} $.

\medskip
\item[$\rm (iii)$] 
$\displaystyle h(a_\mystar/q_\mystar)=3q_1+q_2+\vvq_1-\vq_2$.
\end{itemize}
\end{lemma}

\begin{proof}
(i) and (ii) are plain consequences of \cite[Lemma~5]{ABCZ2025}. 

For (iii), we employ (i) and  $a_2 q_\mystar -a_\mystar q_2 =1$ to get
\begin{equation*}
h(a_\mystar/q_\mystar) =3q_\mystar -q_2 -\vvvq_2 =3q_1+2q_2 -\vvvq_2=3q_1+q_2+\vvq_1-\vq_2 ,
\end{equation*}
where $\vvvq_2=q_\mystar-a_\mystar= q_2+\vq_2-\vvq_1$  denotes the multiplicative inverse
of $ q_2 \pmod{q_\mystar}$ in $[1,q_\mystar)$.
\end{proof}

Assume now that $\cR \subseteq \mathring{\cV}_2$, so  $x_0+y_0 <1$ , $x_0 >y_0$, $3x_0 +2y_0 >1$ and $1+y_0 -2x_0 >0$.
Here, the integer pairs $(q_1,q_2)\in \scrD_Q \cap Q\cR$  must satisfy the inequalities $q_1+q_2 < Q$, $q_2 < q_1$ and $3q_1+2q_2 >Q> 2q_1-q_2$.
Taking \eqref{SQcharacterization} into account, this is equivalent with counting
integer triples $(q_1,q_2,b)$ such that\footnote{Here, $b$ plays the role of $\vq_2=q_1-a_1$, the  multiplicative inverse of $q_2\pmod{q_1}$ in $[1,q_1)$.}
\begin{equation}\label{ineq2}
\begin{split}
& q_1 \in QI_1,\quad q_2\in QI_2, \quad q_2 < q_1, \quad 0<b <q_1,\quad  q_2 b \equiv 1 \pmod{q_1},  \\
&3q_1-q_2-b \le Q , \quad (\Longleftrightarrow h(a_1/q_1)\le Q)\\
& q_1+\bigg( 2-\bigg\lfloor \frac{q_1}{q_2} \bigg\rfloor\bigg) q_2 -\frac{q_2 b -1}{q_1} \le Q ,\quad (\Longleftrightarrow h(a_2/q_2) \le Q)  \\
& 3q_1+2q_2 - b -\frac{q_2  b -1}{q_1} >Q. \quad (\Longleftrightarrow h(a_*/q_*) >Q)
\end{split}
\end{equation}

Consider
\begin{align*}
F_Q (u,v) & := \max \bigg\{ 3u-v -Q, \frac{1+u(u+(2-\lfloor u/v \rfloor)  v-Q)}{v},0 \bigg\} , \\
\widetilde{F}_Q (u,v) & := \max \bigg\{ 3u-v -Q,  \frac{u(u+(2-\lfloor u/v \rfloor\big)  v-Q)}{v},0 \bigg\} , \\
 G_Q(u,v) & :=\min \bigg\{ \frac{1+u(3u +2v -Q)}{u+v}, u \bigg\} ,\\
\widetilde{G}_Q (u,v)  & :=\min \bigg\{ \frac{u(3u +2v -Q)}{u+v}, u \bigg\} ,\quad
 (u,v)\in Q\cR \subseteq Q\mathring{\cV}_2.
\end{align*} 
By a plain check, we see that, for every $ (u,v)\in Q\cR \subseteq Q\mathring{\cV}_2$, we have
\begin{equation}\label{ineq3}
\widetilde{F}_Q(u,v)  \le \max \bigg\{ 3u-v -Q,\frac{u(3v -Q)}{v},0\bigg\} \le \widetilde{G}_Q(u,v) \le u.
\end{equation}
The inclusion $\cR \subseteq \mathring{\cV}_2$ entails  $v\le u$. On the other hand, $\lfloor u/v\rfloor =\lfloor (u/Q)/(v/Q)\rfloor$ takes only
finitely many values (depending on $\cR$ only) whenever $(u,v)\in Q\cR$, and so
\begin{equation*}
F_Q = \widetilde{F}_Q + O_\cR (1) \quad \text{and} \quad G_Q =\widetilde{G}_Q + O(1),
\end{equation*}
uniformly on $Q\cR$. The inequalities in \eqref{ineq2} are equivalent to 
\begin{equation*}
q_2 < q_1 \quad \text{and} \quad  F_Q (q_1,q_2)  \le b < G_Q(q_1,q_2),
\end{equation*} 
hence the number of integer triples $(q_1,q_2,b)$ that satisfy \eqref{ineq2} is given by
\begin{equation*}
S_2 (Q) = \sum\limits_{q_1 \in QI_1}   \sum\limits_{\substack{q_2 \in QI_2 \\ F_Q (q_1,q_2)\le b \le G_Q(q_1,q_2) \\  q_2 b \equiv 1 \pmod{q_1}}} 1
= \sum\limits_{q_1 \in QI_1}  \sum\limits_{\substack{q_2 \in QI_2 \\  \widetilde{F}_Q (q_1,q_2)\le b \le \widetilde{G}_Q(q_1,q_2) \\  q_2 b \equiv 1 \pmod{q_1}}} 1 
+O_\cR (Q).
\end{equation*}
Consider also the sets
\begin{equation*}
\Omega_{Q,u}^{(2)} := \{ (v,w): v\in QI_2, \widetilde{F}_Q (u,v)   \le w \le \widetilde{G}_Q(u,v) \},\quad u \in QI_1,
\end{equation*}
with
\begin{equation*}
\operatorname{Area} (\Omega_{Q,u}^{(2)}) =\int_{QI_2} \big(\widetilde{G}_Q (u,v)-\widetilde{F}_Q(u,v)\big) \, dv.
\end{equation*}
We first apply  \cite[Lemma~3.1]{Boca2007} to the sets $\Omega^{(2)}_{Q,q_1}$, $q_1\in QI_1$, taking $T=\lfloor Q^{1/4}\rfloor$. 
The range of $v$ satisfies $[\alpha,\beta] \subseteq [0,q_1]$. We also have
$V_\alpha^\beta \widetilde{F}_Q(q_1,\cdot) \ll_\cR q_1$, 
$V_\alpha^\beta \widetilde{G}_Q(q_1,\cdot) \ll q_1$, with $\| \widetilde{F} (q_1,\cdot)\|_\infty \le \| \widetilde{G}_Q (q_1,\cdot)\|_\infty \le q_1$,
leading to 
\begin{equation*}
S_2 (Q) = 
\sum\limits_{q_1 \in QI_1} \bigg( \frac{\varphi(q_1)}{q_1^2} \operatorname{Area}(\Omega_{Q,q_1}^{(2)})+O_{\cR,\varepsilon} (\cE_{q_1,\varepsilon}) \bigg),
\end{equation*}
with error terms summing to
\begin{equation*}
\sum\limits_{q_1\in QI_1} \cE_{q_1,\varepsilon} \ll_\cR \sum\limits_{q_1\in QI_1} \bigg(
\frac{q_1}{Q^{1/4}q_1} \, q_1 +Q^{1/4} q_1^{1/2+\varepsilon}\bigg) \ll Q^{7/4+\varepsilon} .   
\end{equation*}
Applying next  \cite[Lemmma~2.3]{BCZ2000} to  $V(u):=\operatorname{Area} (\Omega^{(2)}_{Q,u})/u$,
which is piecewise $C^1$ on $QI_1$ with $\| V\|_\infty\le Qu$, we infer
\begin{align*}
S_2 (Q) & =\frac{1}{\zeta(2)} \int_{QI_1} V(u)\, du +O_{\cR,\varepsilon} (Q^{7/4+\varepsilon}) \\
& =\frac{1}{\zeta(2)} \int_{QI_1} du \int_{QI_2} dv \, \frac{\widetilde{G}_Q(u,v)-\widetilde{F}_Q(u,v)}{u} +O_{\cR,\varepsilon} (Q^{7/4+\varepsilon}) \\
& = \frac{Q^2}{\zeta(2)} \iint_\cR \frac{\widetilde{G}_Q (Qx,Qy)-\widetilde{F}_Q(Qx,Qy)}{Qx} \, dx dy +O_{\cR,\varepsilon} (Q^{7/4+\varepsilon}) \\ 
& =\frac{Q^2}{\zeta(2)} \iint_{\mathcal{R}} H_2(x,y) dx dy+O_{\cR,\varepsilon} (Q^{7/4+\varepsilon}),
\end{align*}
where
\begin{equation}\label{H2}
H_2 (x,y)= \min\bigg\{ \frac{3x+2y-1}{x+y},1\bigg\} -\max \bigg\{ \frac{3x-y-1}{x},\frac{x+(2-\lfloor x/y\rfloor) y-1}{y},0\bigg\} ,
\end{equation}
with $H_2 (x,y) >0$ on $\mathring{\cV}_2$.

Assume next that $\cR \subseteq \mathring{\cV}_3$, so $x_0+y_0 <1$,  $x_0 <y_0$ and $3x_0+2y_0 >1$. 
Here, the integer pairs $(q_1,q_2)\in \scrD_Q \cap Q\cR$  must satisfy the inequalities $q_1+q_2 < Q$, $q_1 <q_2$ and $3q_1+2q_2 >Q$.
Taking~\eqref{SQcharacterization} into account, this is equivalent with counting
integer triples $(q_1,q_2,a)$ such that \footnote{Here, $a$ plays the role of $\vvq_1$, the multiplicative inverse of $q_1 \pmod{q_2}$ in $[1,q_2)$.}  
\begin{equation}\label{ineq1}
\begin{split}
& q_1 \in QI_1,\quad q_2 \in QI_2, \quad q_1 < q_2 , \quad 0<a < q_2, \quad q_1 a \equiv 1 \pmod{q_2} , \\
& \bigg( 2+\bigg\lfloor \frac{q_2}{q_1}\bigg\rfloor \bigg) q_1 -q_2+\frac{q_1 a -1}{q_2} \le Q ,
\quad (\Longleftrightarrow h(a_1/q_1)\le Q)\\
& q_1+q_2+a \le Q ,\quad (\Longleftrightarrow h(a_2/q_2) \le Q)  \\
& 2q_1 +q_2 +a+\frac{q_1 a -1}{q_2} >Q. \quad (\Longleftrightarrow h(a_*/q_*) >Q)
\end{split}
\end{equation}

Consider
\begin{align*}
F_Q (u,v) & :=\max\bigg\{ \frac{1+v (Q-2u-v)}{u+v},0\bigg\} ,\quad
\widetilde{F}_Q (u,v)  :=\max\bigg\{ \frac{v (Q-2u-v)}{u+v},0\bigg\} , \\ 
G_Q(u,v)  & : =\min \bigg\{ \frac{1+v(Q+v-(2+\lfloor v/u\rfloor)u)}{u}, Q-u-v,v\bigg\} ,\\
\widetilde{G}_Q(u,v) & := \min \bigg\{ \frac{v(Q+v-(2+\lfloor v/u\rfloor)u)}{u}, Q-u-v,v\bigg\} , \quad
(u,v)\in Q\cR \subseteq Q \mathring{\cV}_3 .
\end{align*}
By a plain check, we see that
\begin{align*}
& \widetilde{F}_Q(u,v) \le \frac{v (Q-2u)}{u} \le  \frac{v ( Q+v -(2+\lfloor v/u\rfloor) u )}{u}   ,\\
& \widetilde{F}_Q (u,v) \le Q-u -v \quad \text{and} \quad
 \widetilde{F}_Q (u,v) \le v,
\end{align*}
where for the latter we employed the inequality $3x+2y-1 >0$ on $\mathring{\cV}_3$. Thus,
$\widetilde{F}_Q(u,v)\le \widetilde{G}_Q(u,v)\le v$ for all $(u,v)\in Q\cR$. We also notice that $u\le v$ for all $(u,v)\in Q\cR$, 
and also that  $\lfloor v/u\rfloor$ takes only a finite nuumber of values, which depends on $\cR$ only.
The inequalities in \eqref{ineq1} are equivalent to 
\begin{equation}\label{ineq4}
q_1 < q_2 \quad \text{and}\quad F_Q(q_1,q_2) < a\le G_Q(q_1,q_2) .
\end{equation}
Employing inequalities \eqref{ineq4} and the same technique as in~$\cR \subseteq \mathring{\cV}_2$
for the sets
\begin{equation*}
\Omega^{(3)}_{Q,v} :=\{ (u,w): u\in QI_1, \widetilde{F}_Q(u,v)\le w\le \widetilde{G}_Q(u,v)\},\quad v\in QI_2,
\end{equation*}
the number of triples $(q_1,q_2,a)$ as in \eqref{ineq1} is seen to be \mbox{given by}
\begin{align*}
S_3 (Q) & =\sum\limits_{q_2 \in QI_2}  \sum\limits_{\substack{q_1 \in QI_1 \\ F_Q (q_1,q_2) \le a \le G_Q(q_1,q_2) \\ q_1 a \equiv 1 \pmod{q_2}}} 1
=\sum\limits_{q_2 \in QI_2}  \sum\limits_{\substack{q_1 \in QI_1 \\   \widetilde{F}_Q (q_1,q_2) \le a \le \widetilde{G}_Q(q_1,q_2) \\ q_1 a \equiv 1 \pmod{q_2}}} 1 
+O_\cR (Q)  \\ &
=\sum\limits_{q_2 \in QI_2}  \frac{\varphi(q_2)}{q_2^2} \operatorname{Area} (\Omega^{(3)}_{Q,q_2}) 
+O_{\cR,\varepsilon} (Q^{7/4+\varepsilon}) \\
& =\frac{1}{\zeta(2)} \int_{QI_2} dv \int_{QI_1} du\, \frac{\widetilde{G}_Q(u,v)-\widetilde{F}_Q(u,v)}{v}
+O_{\cR,\varepsilon} (Q^{7/4+\varepsilon})         \\
& = \frac{Q^2}{\zeta(2)} \iint_\cR \frac{\widetilde{G}_Q(Qx,Qy)-\widetilde{F}_Q(Qx,Qy)}{Qy}\, dx dy +O_{\cR,\varepsilon} (Q^{7/4+\varepsilon}) \\
&   = \frac{Q^2}{\zeta(2)} \iint_\cR H_3(x,y)\, dx dy +O_{\cR,\varepsilon}  (Q^{7/4+\varepsilon}) ,
\end{align*}
where
\begin{equation}\label{H3}
H_3 (x,y)= \min\bigg\{ \frac{1+y-(2+\lfloor y/x\rfloor)x}{x}, \frac{1-x-y}{y},1\bigg\} - \max\bigg\{ \frac{1-2x-y}{x+y},0\bigg\} ,
\end{equation}
with $H_3 (x,y) >0$ on $\mathring{\cV}_3$. This completes the proof of Theorem \ref{Thm2}.

\begin{remark}
{\em Employing \eqref{SQcharacterization}, \eqref{ineq1} and \eqref{ineq2} to characterize consecutive
elements in $\SF_Q$, one can analyze in a different way than in \cite{ABCZ2025} the gap distribution
in $\SF_Q$.}
\end{remark}

\begin{remark}
{\em In the case $\cR \subseteq \mathring{\cV}_1$, the bound $O_\varepsilon (Q^{7/4+\varepsilon})$ can be improved to $O_\varepsilon (Q^{3/2+\varepsilon})$ by employing
\cite[Lemma~2]{Ust2013} (see also \cite[Lemma~20]{Sis2022}) instead of \cite[Lemma~3.1]{Boca2007}.}
\end{remark}

\section{\texorpdfstring{A partition of the region $\cV$}{A partition of the region cV}}\label{partition}
Here, we assume that the fractions 
$a_1/q_1 < \cdots < a_{r+1}/q_{r+1}$ are consecutive elements in $\cF_Q$ with
$a_1/q_1 , a_{r+1} /q_{r+1} \in \SF_Q$ and $a_2/q_2 ,\ldots ,a_r/q_r \notin \SF_Q$.

$\bullet$ When $r=1$, $a_1/q_1 <a_2/q_2$ are consecutive in both $\SF_Q$ and $\cF_Q$. 
Lemma \ref{L4} provides
\begin{equation*}
\bigg( \frac{q_1}{Q}, \frac{q_2}{Q} \bigg) \in  \cW_1:=\cV_1 .
\end{equation*}

$\bullet$ When $r=2$, Lemma~\ref{L5} provides $K^F_1 (\nu_2)=\nu_2=1$, so $q_2=q_1+q_3$. We gather
$q_1+q_3 \le Q$, $2q_1+q_3 = q_1+q_2 >Q$, $ q_1+ 2q_3 = q_2+q_3 >Q$. Here, $q_1 <q_2$, so 
we also have $2q_1 -q_3 =3q_1 -q_2 \le (2+[q_2/q_1]) q_1 -q_2 \le Q$. These inequalities translate into
\begin{equation*}
\bigg( \frac{q_1}{Q},\frac{q_3}{Q}\bigg) \in  \cW_2:=\bigg\{ (x,y)\in [0,1]^2 : \max\bigg\{ 1-2x,\frac{1-x}{2} ,2x-1\bigg\} \le y \le 1-x\bigg\}  ,
\end{equation*}
with $\cW_2 \subseteq [0,2/3] \times [1/5,1]$.

$\bullet$ When $r=3$, Lemma \ref{L5} provides $K^F_2 (\nu_2,\nu_3)=\nu_2 \nu_3 -1 =1$, so 
$(\nu_2,\nu_3) \in \{ (2,1),(1,2)\}$.

{\bf Case 3.1.} $(\nu_2,\nu_3)=(2,1)$, so $q_3=2q_2 -q_1$, $q_4=q_3-q_2=q_2-q_1$, providing
$q_2=q_1+q_4$, $q_3 =q_1 +2q_4 \le Q$, $q_1+q_2 =2q_1 +q_4 >Q$, and $q_3+q_4 =q_1+3q_4 >Q$.
We also have $q_1 < q_2$ and $2q_1 -q_4 =3q_1 -q_2 \le (2+\lfloor q_2/q_1\rfloor) q_1 -q_2 \le Q$, thus
\begin{equation*}
\bigg( \frac{q_1}{Q},\frac{q_4}{Q}\bigg) \in \cW_{3,2} :=\bigg\{ (x,y) \in [0,1]^2: \max\bigg\{ 1-2x,\frac{1-x}{3} ,2x-1\bigg\} \le y \le \frac{1-x}{2} \bigg\}  ,
\end{equation*}
with $\cW_{3,2} \subseteq [1/3,3/5] \times [1/7,1/3]$.

{\bf Case 3.2.} $(\nu_2,\nu_3)=(1,2)$, so $q_3=q_2-q_1$, $q_4=2q_3-q_2=q_2 -2q_1$, providing
$q_2=2q_1+q_4 \le Q$, $q_3=q_1+q_4$, $q_1+q_2 = 3q_1+q_4> Q$, 
$q_3 +q_4 =q_1 +2q_4 >Q$, thus
\begin{equation*}
\bigg( \frac{q_1}{Q},\frac{q_4}{Q} \bigg)  \in \cW_{3,3} :=\bigg\{ (x,y)\in [0,1]^2 : \max \bigg\{ 1-3x, \frac{1-x}{2} \bigg\} \le y \le 1-2x \bigg\} ,
\end{equation*}
with $\cW_{3,3} \subseteq [0,1/3] \times [1/3,1]$.

$\bullet$ When $r=4$, Lemma \ref{L5} provides $K^F_3 (\nu_2,\nu_3,\nu_4)=\nu_2 \nu_3 \nu_4 -\nu_2-\nu_4 =1$, with five 
positive integer solutions: $(1,2,2)$, $(1,3,1)$, $(2,2,1)$, $(2,1,3)$, $(3,1,2)$. The solution $(2,1,3)$ cannot occur here because 
$q_3=2q_2-q_1$, $q_4=q_3-q_2=q_2-q_1$, $q_5=3q_4-q_3 =q_2-2q_1$ would lead to
$Q\ge q_3>2q_2-3q_1=q_4+q_5 >Q$, a contradiction. The solution $(3,1,2)$ cannot occur either because 
$q_3=3q_2-q_1$, $q_4=q_3-q_2=2q_2-q_1$, $q_5=2q_4-q_3 = q_2-q_1$ would lead to 
$Q\ge q_3 > 3q_2 -2q_1=q_4+q_5 >Q$, a contradiction.
 We are  left with three cases:

{\bf Case 4.1.} $(\nu_2,\nu_3,\nu_4)=(2,2,1)$, so $q_3=2q_2-q_1$, $q_4=2q_3-q_2=3q_2-2q_1$, $q_5=q_4-q_3=q_2-2q_1$, providing
$q_2=q_1+q_5$, $q_3=q_1+2q_5$, $q_4=q_1+3q_5 \le Q$, and thus 
$q_1+q_2 =2q_1+q_5 >Q$ and $q_4+q_5 =q_1+4q_5 >Q$. Again, the inequality $q_5 \ge 2q_1 -Q$ follows for free since 
$q_1 <q_2$ and $2q_1-q_5 =3q_1 -q_2 \le (2+\lfloor q_2/q_1\rfloor)q_1 -q_2 \le Q$. Thus
\begin{equation*}
\bigg( \frac{q_1}{Q}, \frac{q_5}{Q}\bigg)  \in \cW_{4,2}:=\bigg\{ (x,y)\in [0,1]^2 : \max 
\bigg\{ 1-2x, \frac{1-x}{4},2x-1\bigg\} \le y \le \frac{1-x}{3}  \bigg\} ,
\end{equation*}
with $\cW_{4,2} \subseteq  [2/5,4/7] \times [ 1/9,1/5]$.

{\bf Case 4.2.} $(\nu_2,\nu_3,\nu_4)=(1,3,1)$, so $q_3=q_2-q_1$, $q_4=3q_3 -q_2 =2q_2 -3q_1$, $q_5=q_4-q_3=q_2-2q_1$, providing 
$q_2=2q_1+q_5 \le Q$, $q_3=q_1+q_5$, $q_4=q_1+2q_5 \le Q$. Thus
$q_1+q_2 =3q_1+q_5 >Q$ and $q_4+q_5 =q_1+3q_5 >Q$, leading to 
\begin{equation*}
\bigg( \frac{q_1}{Q},\frac{q_5}{Q} \bigg) \in \cW_{4,3}:=\bigg\{ (x,y)\in [0,1]^2: 
\max\bigg\{ 1-3x,\frac{1-x}{3} \bigg\} \le y \le \min \bigg\{ \frac{1-x}{2}, 1-2x\bigg\} \bigg\} ,
\end{equation*}
with $\cW_{4,3} \subseteq [1/5,2/5]\times [1/5,2/5]$.

{\bf Case 4.3.} $(\nu_2,\nu_3,\nu_4)=(1,2,2)$, so $q_3=q_2-q_1$, $q_4=2q_3-q_2 =q_2-2q_1$, $q_5=2q_4-q_3=q_2-3q_1$, providing
$q_2=3q_1+q_5\le Q$, $q_3=2q_1+q_5$, $q_4=q_1+q_5$. Thus $q_1+q_2 =3q_1+q_5 >Q$ and $q_4+q_5 =q_1 +2q_5 >Q$, leading to 
\begin{equation*}
\bigg( \frac{q_1}{Q},\frac{q_5}{Q}\bigg) \in \cW_{4,4}:=\bigg\{ (x,y)\in [0,1]^2:
\max \bigg\{ 1-4x,\frac{1-x}{2} \bigg\} \le y \le 1-3x\bigg\} ,
\end{equation*}
with  $\cW_{4,4}\subseteq [0,1/5] \times [2/5,1]$.

$\bullet$ When $r=5$, Lemma \ref{L5} provides $K^F_4 (\nu_2,\nu_3,\nu_4,\nu_5)=\nu_2\nu_3\nu_4 \nu_5 -\nu_2 \nu_3 -\nu_2\nu_5 - \nu_4\nu_5 +1=1$, with
twelve positive solutions:$(1,3,1,3)$,  $(1,4,1,2)$, $(2,1,3,2)$, $(2,1,4,1)$, $(2,2,1,4)$, $(3,1,2,3)$, $(3,1,3,1)$, 
$(3,2,1,3)$, $(2,2,2,1)$, $(1,3,2,1)$, $(1,2,3,1)$, $(1,2,2,2)$. The first eight solutions are not admissible because 

$\bullet$ $(1,3,1,3)$ would yield $Q \ge q_4=2q_2-3q_1 >2q_2-5q_1 =    q_5+q_6 > Q$, a contradiction.

$\bullet$ $(1,4,1,2)$ would yield $Q\ge q_4=3q_2-4q_1 > 3q_2-5q_1 =q_5 +q_6  >Q$, a contradiction.

$\bullet$ $(2,1,3,2)$ would yield $Q\ge q_3=2q_2-q_1 >  2q_2 -3q_1=q_4+q_5 >Q$, a contradiction.

$\bullet$ $(2,1,4,1)$ would yield $q_3=2q_2-q_1$, $q_4=q_3-q_2 =q_2-q_1$,
$q_5=2q_2-3q_1$, and $q_6=q_5-q_4 =q_2-2q_1 >0$, so $q_2 >2q_1$, which gives $\nu_2=1$, a contradiction.

$\bullet$ $(2,2,1,4)$ would yield $Q\ge q_3=2q_2-q_1 > 2q_2-3q_1 =q_5+q_6 >Q$, a contradiction.

$\bullet$ $(3,1,2,3)$  would yield  $Q\ge q_4=2q_2-q_1 >2q_2-3q_1=q_5+q_2 >Q$, a contradiction.

$\bullet$ $(3,1,3,1)$ would yield $q_2 >q_1$, so $\nu_2 \in \{ 1,2\}$, a contradiction.

$\bullet$ $(3,2,1,3)$ would yield $Q\ge q_3 =3q_2 -q_1 > 3q_2-2q_1 =q_5+q_6 >Q$, a contradiction.

We are left with four cases:

{\bf Case 5.1.} $(\nu_2,\nu_3,\nu_4,\nu_5)=(2,2,2,1)$, so $q_3=2q_2-q_1$, $q-4=2q_3-q_2=3q_2-2q_1$ $q_5=2q_4-q_3$, $q_6=q_5-q_4=q_2-q_1$,
providing $q_2=q_1+q_6$, $q_3 =q_1+2q_6$, $q_4=q_1+3q_6$, $q_5=q_1+4q_6 \le Q$. Thus
$q_1 +q_2 =2q_1 +q_6 >Q$ and $q_1 +5q_6=q_5+q_6 >Q$, leading to
\begin{equation*}
\bigg( \frac{q_1}{Q},\frac{q_6}{Q}\bigg) \in \cW_{5,2}:= \bigg\{ (x,y)\in [0,1]^2 :
\max\bigg\{ 1-2x, \frac{1-x}{5}, 2x-1\bigg\} \le y \le \frac{1-x}{4}\bigg\} ,
\end{equation*}
with $\cW_{5,2} \subseteq [3/7,6/11] \times [1/11,1/7]$.

{\bf Case 5.2.} $(\nu_2,\nu_3,\nu_4,\nu_5)=(1,3,2,1)$, so $q_3=q_2-q_1$, $q_4=3q_3-q_2=2q_2-3q_1$, 
$q_5=2q_4-q_3=3q_2-5q_1$, $q_6=q_5-q_4=q_2-2q_1$, providing $q_2=2q_1+q_6 \le Q$, $q_3 =q_1+q_6$, 
$q_4=q_1+2q_6$, $q_5=q_1+3q_6 \le Q$. Thus $q_1+q_2=2q_1+q_6 >Q$ and $q_5+q_6 =q_1+4q_6 >Q$, leading to
\begin{equation*}
\bigg( \frac{q_1}{Q},\frac{q_6}{Q}\bigg) \in \cW_{5,3}:=\bigg\{ (x,y)\in [0,1]^2 :
\max\bigg\{ 1-3x,\frac{1-x}{4}\bigg\} \le y\le \min \bigg\{ \frac{1-x}{3}, 1-2x\bigg\} \bigg\} ,
\end{equation*}
with $\cW_{5,3} \subseteq [1/4,3/7] \times [1/7,1/4]$.

{\bf Case 5.3.} $(\nu_2,\nu_3,\nu_4,\nu_5) =(1,2,3,1)$, so $q_3=q_2-q_1$, $q_4=2q_3-q_2=q_2-2q_1$,
$q_5=3q_4-q_3=2q_2-5q_1$, $q_6=q_5-q_4=q_2-3q_1$, providing $q_2=3q_1+q_6 \le Q$, $q_3=2q_1+q_6$, 
$q_4=q_1+q_6$, $q_5=q_1+2q_6 \le Q$. Thus $q_1+q_2=4q_1+q_6 >Q$ and $q_5+q_6 =q_1+3q_6 >Q$, leading to
\begin{equation*}
\bigg( \frac{q_1}{Q},\frac{q_6}{Q}\bigg) \in \cW_{5,4}:=\bigg\{ (x,y)\in [0,1]^2 :
\max \bigg\{ 1-4x, \frac{1-x}{3} \bigg\} \le y \le \min \bigg\{ \frac{1-x}{2}, 1-3x \bigg\} \bigg\} ,
\end{equation*}
with $\cW_{5,4} \subseteq [1/7,1/4] \times [1/4,3/7]$.

{\bf Case 5.4.}  $(\nu_2,\nu_3,\nu_4,\nu_5)=(1,2,2,2)$, so $q_3=q_2-q_1$, $q_4=2q_3-q_2=q_2-2q_1$, $q_5=2q_4-q_3=q_2-3q_1$.
$q_6=2q_5-q_4=q_2-4q_1$, providing $q_2=4q_1+q_6 \le Q$, $q_3=3q_1+q_6$, $q_4=2q_1+q_6$, 
$q_5=q_1+q_6$. Thus $q_1+q_2 =5q_1+q_6 >Q$ and $q_5+q_6 =q_1+2q_6 >Q$, leading to
\begin{equation*}
\bigg( \frac{q_1}{Q},\frac{q_6}{Q}\bigg) \in \cW_{5,5}:=\bigg\{ (x,y)\in [0,1]^2 :
\max \bigg\{ 1-5x,\frac{1-x}{2} \bigg\} \le y\le 1-4x \bigg\} ,
\end{equation*}
with $\cW_{5,5} \subseteq [0,1/7] \times [3/7,1]$.

Next, we show an example where the boundary line $3x+2y=1$ of $\cV$
occurs as a side of a region $\cW_{r,i}$.
The tuple $(q_1,\ldots,q_{10})=(27,95, 68, 41, 96, 55, 69,83,97,14)$ represents consecutive elements in $\cF_{100}$ 
with $q_1=27$ and $q_{10}=14$ consecutive in $\SF_{100}$ and mediant $q_1+q_{10}=41=q_4$.
The equalities $q_2= 3q_1+q_{10}$, $q_3=2q_1+q_{10}$, $q_4=q_1+q_{10}$, $q_5=2q_1+3q_{10}$,  $q_6=q_1+2q_{10}$, $q_7=q_1+3q_{10}$,
$q_8 =q_1+4q_{10}$, $q_9=q_1+5q_{10}$ show that the region $Q \cW_{9,4}$ is defined by the inequalities 
$\max\{ 3q_1+q_{10}, 2q_1+3q_{10} , q_1 +5q_{10}  \} \le Q < \min\{ q_1+q_2 =4q_1+q_{10}, q_3+q_4=3q_1+2q_{10}, 
q_6+q_7=2q_1+5q_{10}, q_9+q_{10}=q_1+6q_{10}\}$, and thus
\begin{equation*}
\cW_{9,4}=\{ h_1(x) \le y \le h_2(x) \} ,
\end{equation*}
with 
\begin{equation*}
\begin{split}
& h_1(x)=  \max\bigg\{ 1-4x,\frac{1-3x}{2},\frac{1-2x}{5} ,\frac{1-x}{6}\bigg\}
=\begin{cases} 1-4x & \mbox{\rm if $0\le x\le 1/5$} \\
(1-3x)/2 & \mbox{\rm if $1/5 \le x \le 1/4$} \\
(1-x)/6 & \mbox{\rm if $1/4 \le x \le 1/2$}
\end{cases} ,  \\
& h_2(x) =  \min \bigg\{ 1-3x,\frac{1-2x}{3} ,\frac{1-x}{5}\bigg\} =\begin{cases}
(1-x)/5 & \mbox{\rm if $0\le x \le 2/7$} \\
1-3x & \mbox{\rm if $2/7 \le x \le 1/2$}
\end{cases}  .
\end{split}
\end{equation*}
We have $\cW_{9,4} \subseteq [3/13,5/17] \times [  2/17  ,2/13]$ and the line $3x+2y=1$ is a side of $\cW_{9,4}$ because 
\begin{equation*}
\cW_{9,4} \cap \bigg( \bigg[ \frac{1}{5},\frac{1}{4} \bigg]   \times [0,1]\bigg) 
=\bigg\{ (x,y) : \frac{1}{5} \le x \le \frac{1}{4}, \frac{1-3x}{2} \le y \le \frac{1-x}{5}\bigg\}.
\end{equation*}
The region $\cW_{9,4}$ is  illustrated in the right-hand picture from Figure \ref{FigureVandWs}.

\section{\texorpdfstring{An index identity and the $\Phi$ function}{An index identity and the Phi function}}\label{Sect5}

The proof of Proposition \ref{Prop3} follows closely the idea of \cite[Theorem 1]{HS2003} and 
proceeds by induction on $Q$.
As $a_{i+1} q_i -a_i q_{i+1}=1$, we have
\begin{equation*}
\frac{q_{i-1}+q_{i+1}}{q_i}=\frac{a_{i-1}+a_{i+1}}{a_i} 
=a_{i+2} q_i -a_i q_{i+2} \in \mathbb{N},
\qquad i=1,\ldots , S_Q.
\end{equation*}
Denote 
\begin{equation*}
T_Q:=\sum\limits_{i=1}^{S_Q} \frac{q_{i-1}+q_{i+1}}{q_i} .
\end{equation*}
For $Q=3$, we have $\SF_3=\{ 1/1\}$, $S_3=1$, and $T_3=2=3S_3-1$.
Suppose that $T_{Q-1}=3S_{Q-1} -1$ for some $Q\ge 4$. The number of mediant insertions 
when enlarging $\SF_{Q-1}$ to $\SF_Q$ is precisely $S_Q -S_{Q-1}$. When such a
mediant insertion $\gamma_\mystar=a_\mystar/q_\mystar$ occurs,
say between $\gamma_i$ and $\gamma_{i+1}$, its contribution to $T_{Q-1}$ is
\begin{equation*}
A_i  =\frac{q_{i-1}+q_{i+1}}{q_i} +\frac{q_i+q_{i+2}}{q_{i+1}} ,
\end{equation*}
while its contribution to $T_Q$ is
\begin{equation*}
\begin{split}
B_i &   =\frac{q_{i-1}+q_*}{q_i} +\frac{q_i+q_{i+1}}{q_*} +\frac{q_*+q_{i+2}}{q_{i+1}}    \\
&  =  \frac{q_{i-1}+(q_i+q_{i+1})}{q_i} +1
+\frac{(q_i+q_{i+1})+q_{i+2}}{q_{i+1}} =A_i+3.
\end{split}
\end{equation*}
This shows that $T_Q =T_{Q-1} +3(S_Q-S_{Q-1})=3S_Q-1$, as desired.

\begin{figure}[t]
\centering
\hfill
\includegraphics[width=0.328\textwidth]{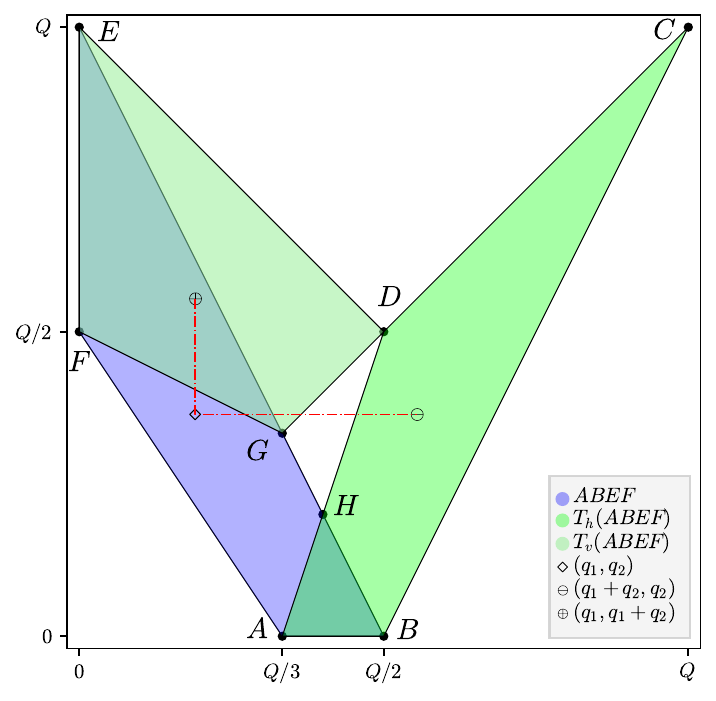} 
\hfill\mbox{}
\caption{
The horizontal and vertical transformations when insertions are made.
Thus, a point $(q_1,q_2)\in\scrD_{Q}$ that disappears at level $Q$ 
is replaced in $\scrD_{Q+1}$ with two new points $(q_1+q_2,q_2)$ and $(q_1,q_1+q_2)$. 
Therefore, through insertions, the polygon $ABEF$, in which the points 
that will disappear at level $Q$ are found, 
transforms into the polygons $ABCD$ and $GDEF$ that will contain the new points.
}
\label{FigureTransformationThTv}
\end{figure}

This also suggests that an analogue of Euler's function $\varphi (Q)$ is given here by
\begin{equation*}
\begin{split}
\Phi (Q) & := S_Q -S_{Q-1} =\# \bigg\{ \frac{a}{q} \in \QQ \cap (0,1] : q+a+\va =Q\bigg\}   \\
& =\# \bigg\{ \frac{a_1}{q_1} < \frac{a_2}{q_2} \text{ consecutive in $\SF_{Q-1}$:} \
h\bigg( \frac{a_1+a_2}{q_1+q_2} \bigg) = Q \bigg\} \\
& =\# \left\{ (a_1,q_1,a_2,q_2) \in \NN^4 : 
\begin{matrix} a_2 q_1-a_1 q_2 =1,  \  a_i < q_i, \\
 \max\big\{ h \big( \frac{a_1}{q_1}\big),
h \big( \frac{a_2}{q_2}\big) \big\} \le Q-1 , \   h\big( \frac{a_1+a_2}{q_1+q_2}\big) =Q 
\end{matrix} \right\} .
\end{split}
\end{equation*}

The arithmetic function $\Phi$ has an erratic behavior (see Figure~\ref{FigureGraphOfPhi}). It is not multiplicative and 
\begin{equation*}
\begin{split}
& \Phi (3)= S_3 =1,\quad \Phi(4)=1 ,\quad \Phi (5)=1,\quad
\Phi (6) =1  ,\quad \Phi (7)=2, \quad \Phi (8)=1, \\
& \Phi(9)=1,\quad \Phi(10)=4, \quad \Phi (11)= 1, \quad \Phi (12)=1,\quad
\Phi (13)= 4, \quad \Phi (14)=2,  \\
& \Phi(15)= 3,\quad \Phi (16) = 4, \quad \Phi(17) =1, \quad \Phi(18)= 4,
\quad \Phi(19) = 4,\quad \Phi(20) =5.
\end{split}
\end{equation*}

\begin{figure}[thb]
\centering
\hfill
 \includegraphics[width=0.48\textwidth]{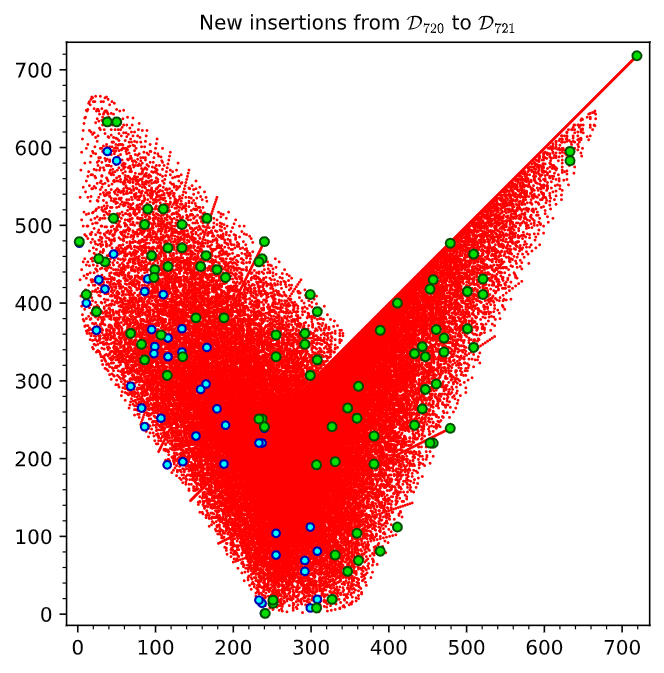}
 \includegraphics[width=0.48\textwidth]{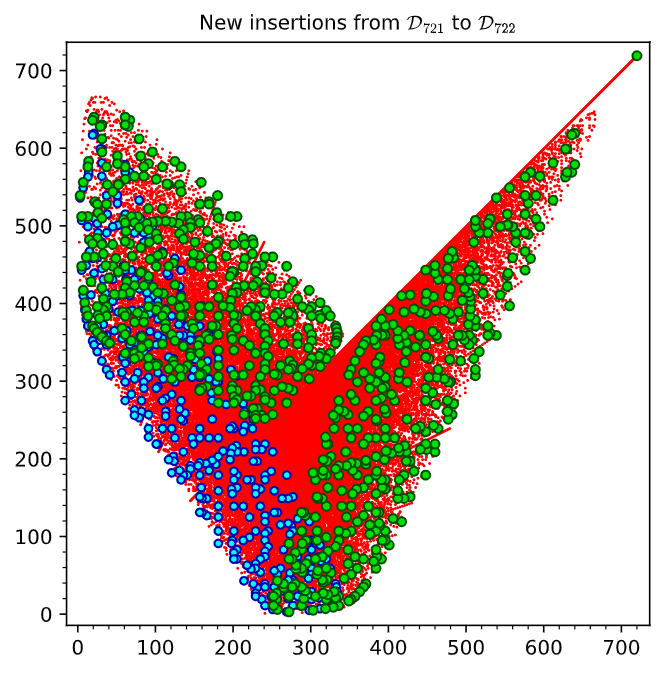}
\hfill\mbox{}\\
\hfill
\includegraphics[width=0.48\textwidth]{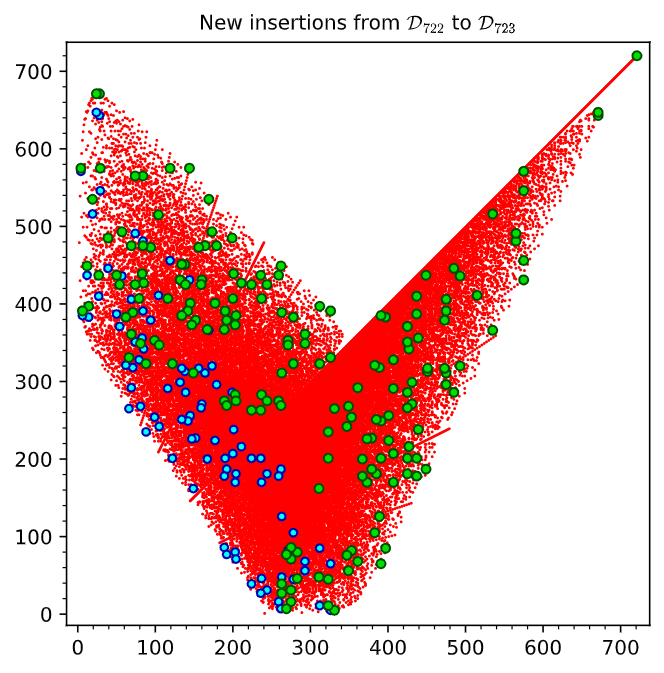}
\includegraphics[width=0.48\textwidth]{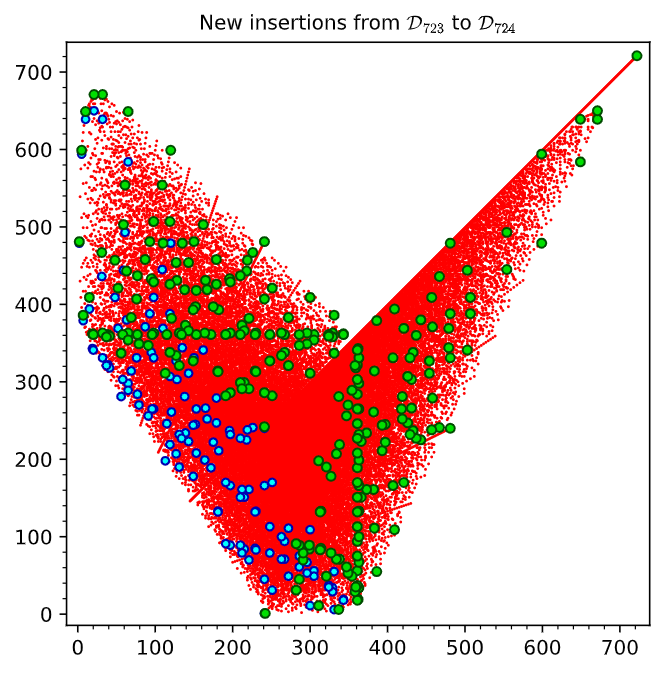}
\hfill\mbox{}
\caption{The points that disappear and the new points created by insertion when
passing from level $Q-1$ to level $Q$ for $Q=721,722,723$, and $724$.
}
\label{FigureGoneAndBuds}
\end{figure}

In Figures~\ref{FigureGraphOfPhi},~\ref{FigureTransformationThTv},  and~\ref{FigureGoneAndBuds}, it can be observed that not only does 
$\Phi(Q)$ exhibit significant oscillation, 
but also that at times, particularly around levels $Q$ 
with a large number of divisors,
the disappearing and newly created points tend to group together 
into distinct formations.



\mbox{}
\end{document}